\documentclass[smallcondensed]{svjour3}     

\usepackage[dvipsnames, table]{xcolor}

\usepackage[T1]{fontenc}
\usepackage[latin1]{inputenc}

\usepackage{amsmath, amsfonts, amssymb, mathtools, bbm}
\smartqed  

\usepackage{graphicx}
\usepackage[caption=false,font=footnotesize]{subfig}
\usepackage{lmodern}
\usepackage{microtype}
\usepackage{algorithm}
\usepackage{booktabs}
\usepackage{multirow}
\usepackage{cite}
\usepackage[multiple]{footmisc}
\usepackage{theorem}

\DeclareMathOperator*{\minimize}{minimize\;}
\DeclareMathOperator*{\maximize}{maximize\;}
\DeclareMathOperator{\prox}{prox}
\DeclareMathOperator{\Id}{Id}
\DeclareMathOperator{\sign}{sign}

\newcommand{\NN}{\mathbb{N}}
\newcommand{\RR}{\mathbb{R}}
\newcommand{\LL}{\mathbb{L}}

\newcommand{\HH}{\mathcal{H}}
\newcommand{\GG}{\mathcal{G}}
\newcommand{\EE}{\mathsf{E}}
\newcommand{\PP}{\mathsf{P}}

\newcommand{\HHH}{{\ensuremath{\boldsymbol{\mathcal{H}}}}}
\newcommand{\GGG}{{\ensuremath{\boldsymbol{\mathcal{G}}}}}

\newcommand{\rr}{\boldsymbol{r}}

\newcommand{\sv}{\boldsymbol{s}}
\newcommand{\tw}{\boldsymbol{t}}

\newcommand{\vv}{\boldsymbol{v}}
\newcommand{\ww}{\boldsymbol{w}}
\newcommand{\xx}{\boldsymbol{x}}
\newcommand{\zz}{\boldsymbol{z}}
\newcommand{\TT}{\boldsymbol{A}}
\newcommand{\AAA}{\boldsymbol{A}}
\newcommand{\ff}{\boldsymbol{f}}
\newcommand{\hh}{\boldsymbol{h}}

\theoremstyle{plain}{\theorembodyfont{\rmfamily}%
}

\renewcommand\theenumi{(\roman{enumi})}
\renewcommand\theenumii{\alph{enumii})}

\begin{document}

\title{%
A Random Block-Coordinate Douglas-Rachford Splitting Method with Low Computational Complexity for Binary Logistic Regression%
\thanks{This work was partly supported by the the CNRS MASTODONS project under grant 2016TABASCO.}
}

\titlerunning{Random Douglas-Rachford Splitting with Low Computational Complexity}

\author{%
Luis~M.~Brice{\~n}o-Arias\and%
Giovanni~Chierchia\and%
Emilie~Chouzenoux\and%
Jean-Christophe~Pesquet%
}

\institute{%
L.~M.~Brice{\~n}o-Arias \at Departamento de Matem\'atica, Universidad T\'ecnica Federico Santa Mar\'ia, Av Espan\~a 1681, Valpara\'iso, Chile.
\and
G. Chierchia (corresponding author) and E. Chouzenoux \at Universit\'e Paris Est, LIGM, CNRS UMR 8049, ESIEE Paris, UPEM, Noisy-le-Grand, France.
\and
E. Chouzenoux and J.-C. Pesquet \at Center for Visual Computing, INRIA Saclay, CentraleSup\' elec, University Paris-Saclay, Gif sur Yvette, France.
}

\maketitle

\begin{abstract}
In this paper, we propose a new optimization algorithm for sparse logistic regression based on a stochastic version of the Douglas-Rachford splitting method. Our algorithm sweeps the training set by randomly selecting a mini-batch of data at each iteration, and it allows us to update the variables in a block coordinate manner.
Our approach leverages the proximity operator of the logistic loss, which is expressed with the 
generalized Lambert W function.
Experiments carried out on standard datasets demonstrate the efficiency of our approach w.r.t.\ stochastic gradient-like methods.%
\end{abstract}

\section{Introduction}
Sparse classification algorithms have gained much popularity in the context of supervised learning, thanks to their ability to discard irrelevant features during the training stage. Such algorithms aim at learning a weighted linear combination of basis functions that fits the training data
, while encouraging as many weights as possible to be equal to zero. This amounts to solving an optimization problem that involves a loss function plus a sparse regularization term. Different types of classifiers arise by varying the loss function, the most popular being the hinge and the logistic losses \cite{Rosasco2004,Bartlett2006}. 

In the context of supervised learning, sparse regularization traces back to the work of Bradley and Mangasarian~\cite{Bradley1998}, who showed that the $\ell_1$-norm can efficiently perform feature selection by shrinking small coefficients to zero. Other forms of regularization have also been studied, such as the $\ell_0$-norm~\cite{Weston2002}, the $\ell_p$-norm with $p > 0$ \cite{Liu2007}, the $\ell_\infty$-norm \cite{Zou2008}, and other nonconvex terms \cite{Laporte2014}. Mixed-norms have been investigated as well, due to their ability to impose a more structured form of sparsity \cite{Krishnapuram2005, Meier2008, Duchi2009, Obozinski2010, Yuan2010, Rosasco2013}.

Many efficient learning algorithms exist in the case of quadratic regularization, by benefiting from the advantages brought by Lagrangian duality \cite{Blondel2014}. This is unfortunately not true for sparse regularization, because the dual formulation is as difficult to solve as the primal one. Consequently, sparse linear classifiers are usually trained through the direct resolution of the primal optimization problem. Among the possible approaches, one can resort to linear programming \cite{Wang2007}, gradient-like methods \cite{Krishnapuram2005, Mairal2013}, proximal algorithms \cite{Laporte2014, Chierchia2015, Barlaud2015}, and other optimization techniques \cite{Tan2010}. 

Nowadays, it is well known that random updates can significantly reduce the computational time when a quadratic regularization is used \cite{Hsieh2008, Lacoste2013}. Therefore, a great deal of attention has been paid recently to stochastic approaches capable of handling a sparse regularization \cite{Pereyra2016}. The list of investigated techniques includes block-coordinate descent strategies \cite{Shalev2011, Blondel2013, Fercoq2015,Lu2015}, stochastic forward-backward iterations \cite{Langford2009, Xiao2010, Richtarik2014, Rosasco2014, Combettes2016}, random Douglas-Rachford splitting methods \cite{Combettes2015}, random primal-dual proximal algorithms 
\cite{Pesquet2015}, and stochastic majorization-minimization methods \cite{Mairal2013b, Chouzenoux2015}. 


In this paper, we propose a random-sweeping block-coordinate Douglas-Rachford splitting method. In addition to the stochastic behavior, it presents three distinctive features with respect to related approaches \cite{Briceno2011, Bot2013, Combettes2015}. Firstly, the matrix to be inverted at the initial step is block-diagonal, while in the concurrent approaches, it did not present any specific structure. The block diagonal property implies that the inversion step actually amounts to inverting a set of smaller size matrices. Secondly, the proposed algorithm can take advantage explicitly from a strong convexity property possibly fulfilled by some of the functions involved in the optimization problem. Finally, the dual variables appear explicitly in the proposed scheme, making it possible to use clever block-coordinate descent strategies \cite{Perekrestenko2017}.

Moreover, the proposed algorithm appears to be well tailored to binary logistic regression with sparse regularization. In contrast to gradient-like methods, our approach deals with the logistic loss through its proximity operator. This results in an algorithm that is not tied up to the Lipschitz constant of the loss function, possibly leading to larger updates per iteration. In this regard, our second contribution is to show that the proximity operator of the binary logistic loss can be expressed in closed form using the generalized Lambert W function \cite{Mezo2014, Maignan2016}. We also provide comparisons with state-of-the-art stochastic methods using benchmark datasets.

The paper is organized as follows. In Section~\ref{sec:optim}, we derive the proposed Douglas-Rachford algorithm. In Section~\ref{sec:logistic}, we introduce sparse logistic regression, along with the proximal operator of the logistic loss. In Section~\ref{sec:four}, we evaluate our approach on standard datasets, and compare it to three sparse classification algorithms proposed in the literature \cite{Xiao2010, Rosasco2014, Chierchia2016}. Finally, conclusions are drawn in Section~\ref{sec:five}.

\vspace{0.5em}

\textsc{Notation}: $\Gamma_0(\HH)$ denotes the set of proper, lower semicontinuous, convex functions from a real Hilbert space $\HH$ to $]-\infty,+\infty]$. 
Let  $\psi\in\Gamma_0(\HH)$. For every $\nu \in \HH$, the subdifferential of $\psi$ at $\nu$ is $\partial \psi(\nu) = \{\xi\in\HH~\mid~(\forall \zeta\in\HH)\; \left\langle \zeta-\nu \mid \xi\right\rangle +\psi(\nu)\leq \psi(\zeta)\}$, the proximity operator of $\psi$ at $\nu$ is $\prox_\psi(\nu)= \operatorname*{argmin}_{\xi\in\HH}\; \frac12 \|\xi-\nu\|^2 + \psi(\nu)$, and the conjugate of $\psi$ is $\psi^*= \sup_{\xi\in\HH} \left\langle \xi\mid\cdot \right\rangle -\psi(\xi)$ in $\Gamma_0(\HH)$. 
The adjoint of a bounded linear operator $A$ from $\HH$ to a real Hilbert space $\GG$ is denoted by $A^*$
Let $(\Omega,\mathcal{F},\mathbb{P})$ be the underlying probability space,
the $\boldsymbol{\sigma}$-algebra generated by a family $\Phi$ of random variables is denoted by $\boldsymbol{\sigma}(\Phi)$.

\section{Optimization method}\label{sec:optim}
Throughout this section, $\HH_1,\ldots,\HH_B$, $\GG_1,\ldots,\GG_L$ are separable real Hilbert spaces. In addition,
$\HHH = \HH_1 \oplus \cdots \oplus \HH_B$ denotes the Hilbertian sum of 
$\HH_1,\ldots,\HH_B$.
Any vector $\boldsymbol{v}\in \HHH$ can thus be uniquely decomposed as $(v_b)_{1\le b 
\le B}$ where, for every $b\in \{1,\ldots,B\}$,
$v_b \in \HH_b$. In the following, a similar notation will be used to denote vectors in any product space (in bold) and their components.

We will now aim at solving the following problem.
\begin{problem}\label{problem}
For every $b\in\{1,\dots,B\}$ and for every $\ell\in\{1,\dots,L\}$, let $f_b \in \Gamma_0(\HH_b)$, let $h_\ell\colon \GG_\ell \to \RR$
be a differentiable convex function with $\beta_\ell$-Lipschitz gradient, for some $\beta_\ell 
\in 
]0,+\infty[$,
and let $A_{\ell,b}$ be a linear bounded operator from $\HH_b$ to $\GG_\ell$.  
The problem is to
\[
\minimize_{\boldsymbol{w} \in\HHH} \; \sum_{b=1}^B f_b(w_b) + \sum_{\ell=1}^L 
h_\ell\Big(\sum_{b=1}^B A_{\ell,b} w_b\Big),
\]
under the assumption that the set of solutions $\mathcal{E}$ is nonempty.
\end{problem}

In order to address Problem~\ref{problem}, we propose to employ the random-sweeping 
block-coordinate version of the Douglas-Rachford splitting method with stochastic errors 
described in Algorithm \ref{algo:bcdg}. Let us define
\begin{equation}
(\forall b\in\{1,\dots,B\})\quad C_b = \Big(\Id+\tau_b \sum_{\ell=1}^L 
\frac{\gamma_\ell}{1+\gamma_\ell \rho_\ell} A_{\ell,b}^* \, A_{\ell,b} \Big)^{-1}\colon\HH_b\to\HH_b,
\end{equation}
where $(\tau_b)_{1\le b \le B}$ and $(\gamma_\ell)_{1\le \ell \le 1}$ are the positive constants introduced in Algorithm \ref{algo:bcdg}.
\begin{algorithm}
\caption{Random Douglas-Rachford splitting for solving Problem \ref{problem}}\label{algo:bcdg}
{\small
\textsc{Initialization}\\[-0.5em]
\[\hspace{-4.5cm}
\left\lfloor
\begin{aligned}
&\textrm{Set $(\tau_b)_{1\le b \le B}\in ]0,+\infty[^B$ and $\eta \in 
]0,1]$.}\\
&\textrm{For every $\ell \in \{1,\ldots,L\}$, set $\rho_\ell \ge 0$ such that  $B\beta_\ell 
\rho_\ell \le 1$.}\\
&\textrm{For every $\ell \in \{1,\ldots,L\}$, set $\gamma_\ell > 0$ such that $\gamma_\ell \rho_\ell<1$.}\\
&\textrm{$(\forall b\in\{1,\dots,B\})\quad u_b^{[0]} = \sum_{\ell=1}^L 
\frac{1}{1+\gamma_\ell\rho_\ell} A_{\ell,b}^* \, s_{\ell,b}^{[0]}$}\\[0.25em]
\end{aligned}
\right.
\]
\noindent \textsc{For}\; $i = 0, 1, \dots$\\[-0.5em]	
\[
\!\!\!\!\!\!\!\!\!\!\!\!\!\!\!\!\!\!\!\!\!\!\!\!\!\!\!\!\!\!\!\!\!\!\!\!\!\!\!\!\!\!\!\!\!\!\!\!\!\!\!\!\!\!\!\!\!\!\!\!\!\!\!\!\!\!\!\!\!\!\!
\left\lfloor
\begin{aligned}
&\textrm{Set}\; \mu^{[i]}\in\left]\eta,2-\eta\right[\\
%
&\textrm{for $b=1,\dots,B$}\\
&\left\lfloor
\begin{aligned}
& w_b^{[i+1]} = w_b^{[i]}+ \varepsilon_b^{[i]}\left(C_b\left(t_b^{[i]}-\tau_b 
u_b^{[i]}\right)-w_b^{[i]}\right) \\
& t_b^{[i+1]} = t_b^{[i]} + \varepsilon_b^{[i]}\mu^{[i]} \left( \prox_{\tau_b 
f_b}(2w_b^{[i+1]}-t_b^{[i]}) + 
a_b^{[i]} - w_b^{[i+1]} \right)
\end{aligned}
\right.\\[0.5em]
&\textrm{for $\ell=1,\dots,L$}\\
&\left\lfloor
\begin{aligned}
&\boldsymbol{v}_\ell^{[i+1]} =\boldsymbol{v}_\ell^{[i]}+ 
\varepsilon_{B+\ell}^{[i]}\left(\frac{\boldsymbol{s}_\ell^{[i]} +\gamma_\ell \big( A_{\ell,b} 
w_b^{[i]}\big)_{1\le b\le B}}{1+\gamma_\ell\rho_\ell}-\boldsymbol{v}_\ell^{[i]}\right)\\
&p_\ell^{[i]} = 2\sum_{b=1}^B v_{\ell,b}^{[i+1]}-\sum_{b=1}^B s_{\ell,b}^{[i]} \\
&q_\ell^{[i]} = \prox_{\frac{B(1-\gamma_\ell\rho_\ell)}{\gamma_\ell} h_\ell}\big(p_\ell^{[i]}/\gamma_\ell\big) + d_\ell^{[i]}\\
&\textrm{for $b=1,\dots,B$}\\
&\left\lfloor
\begin{aligned}
&s_{\ell,b}^{[i+1]} = s_{\ell,b}^{[i]} + \varepsilon_{B+\ell}^{[i]}\mu^{[i]} \Big(
\frac{ p_\ell^{[i]} - \gamma_\ell \, q_\ell^{[i]} }{B(1-\gamma_\ell\rho_\ell)}
-v_{\ell,b}^{[i+1]}\Big)
\end{aligned}
\right.
\end{aligned}
\right.\\[0.5em]
&\textrm{for $b=1,\dots,B$}\\
&\left\lfloor
\begin{aligned}
u_b^{[i+1]} = u_b^{[i]} + \sum_{\ell=1}^L  \frac{\varepsilon_{B+\ell}^{[i]}}{1+\gamma_\ell\rho_\ell} A_{\ell,b}^* 
\,  \big(s_{\ell,b}^{[i+1]}-s_{\ell,b}^{[i]}\big).
\end{aligned}
\right.\\[0.5em]
\end{aligned}
\right.
\]
}
\end{algorithm}
The next result establishes the convergence of the proposed algorithm.
\begin{proposition}\label{prop:algo}
	For every $b\in\{1,\dots,B\}$, let $w_b^{[0]}$, $t_b^{[0]}$ and $(a_{b}^{[i]}\big)_{i\in\NN}$ 
	be 
	$\HH_b$-valued random variables and, for every $\ell\in\{1,\dots,L\}$, let 
	$\boldsymbol{v}_\ell^{[0]}$ and $\boldsymbol{s}_\ell^{[0]}$ be 
	$\GG_\ell^B$-valued random variables and let $(d_{\ell}^{[i]}\big)_{i\in\NN}$ be
	$\GG_\ell$-valued random variables. In addition, let 
	$(\boldsymbol{\varepsilon}^{[i]})_{i\in\NN}$  be identically 
		distributed $\{0,1\}^{B+L}\setminus\{\boldsymbol{0}\}$-valued random variables 
and in Algorithm \ref{algo:bcdg} assume that
\begin{enumerate}
\item \label{a:prop1iv}  $(\forall i\in\NN)$
$\boldsymbol{\sigma}(\boldsymbol{\varepsilon}^{[i]})$ and 
$\boldsymbol{\chi}^{[i]}=\boldsymbol{\sigma}(\boldsymbol{t}^{[0]},\dots,\boldsymbol{t}^{[i]},\boldsymbol{s}^{[0]},
\dots,\boldsymbol{s}^{[i]})$
are independent;
\item \label{a:prop1i} $(\forall 
b\in\{1,\dots,B\})\;\sum_{i\in\NN}\sqrt{\EE(\|a_{b}^{[i]}\|^2|\boldsymbol{\chi}^{[i]})} < 
+\infty$; 
\item  \label{a:prop1ii} 
$(\forall\ell\in\{1,\dots,L\})\;\sum_{i\in\NN}\sqrt{\EE(\|d_{\ell}^{[i]}\|^2|\boldsymbol{\chi}^{[i]})}
 < +\infty$;
\item\label{prop:algo:ass_prob} $(\forall b\in\{1,\dots,B\})\; \PP[\varepsilon_{b}^{[0]}=1] > 0$ and $(\forall\ell\in\{1,\dots,L\})\; \PP[\varepsilon_{B+\ell}^{[0]}=1] > 0$.
\end{enumerate}
Then, the sequence $(\ww^{[i]})_{i\in\NN}$ generated by Algorithm~\ref{algo:bcdg} 
converges weakly $\PP$-a.s. to an $\mathcal{E}$-valued 
random variable.
\end{proposition}

\begin{proof}
Problem~\ref{problem} can be reformulated as minimizing
$
\ff + \hh\circ \boldsymbol{A}
$
where 
\begin{align}
&
\label{e:fff}\ff\colon\HHH\to]-\infty,+\infty]\colon \ww\mapsto \sum_{b=1}^B f_b(w_b)\\
&\label{e:defA}\boldsymbol{A}\colon \HHH \to \GGG\colon \ww \mapsto 
\left(A_{\ell,1}w_1,\ldots,A_{\ell,B}w_B\right)_{1\le\ell\le L}\\
\label{e:hhh}
&\hh\colon\GGG\to\RR\colon\vv\mapsto 
\sum_{\ell=1}^Lh_\ell(\boldsymbol{\Lambda}_{\ell}\boldsymbol{v}_{\ell} 
)\\
(\forall \ell\in\{1,\ldots,L\})\quad &\boldsymbol{\Lambda}_{\ell}\colon 
\GG_{\ell}^B\to\GG_{\ell}\colon\boldsymbol{v}_{\ell}\mapsto \sum_{b=1}^Bv_{\ell,b}
\end{align}
and $\boldsymbol{v}=(\boldsymbol{v}_{\ell})_{1\leq\ell\le L}$ denotes a generic element of 
$\GGG = \GG_1^B\oplus \cdots \oplus \GG_L^B$  with 
$\boldsymbol{v}_{\ell}=(v_{\ell,b})_{1\leq b\le B}\in\GG_{\ell}^B$ for every 
$\ell\in\{1,\ldots,L\}$.
Since ${\rm dom}(\hh)=\GGG$, from \cite[Theorem~16.47(i)]{Bauschke2017},
Problem~\ref{problem} is equivalent to 
\begin{equation}\label{eq:first_inclusion}
\text{find}\quad\ww\in \HHH\quad\text{such that}\quad  \boldsymbol{0} \in \partial 
\ff(\ww) + \AAA^* 
\nabla\hh\left(\AAA \ww \right),
\end{equation}
which, from \cite[Proposition~2.8]{Briceno2011} is also equivalent to
\begin{equation}\label{eq:old_inclusion}
\text{find}\quad(\ww,\vv) \in \HHH\times \GGG\quad\text{such that 
}\quad(\boldsymbol{0},\boldsymbol{0}) \in 
\boldsymbol{N}(\ww,\vv)+ 
\boldsymbol{S}(\ww,\vv),
\end{equation}
where
$\boldsymbol{N}\colon (\ww,\vv)\mapsto \partial \ff(\ww)\times 
\partial \hh^*(\vv) $ is maximally monotone and $\boldsymbol{S}\colon 
(\ww,\vv)\mapsto(\AAA^*\vv,-\AAA\ww)$ is a skewed linear operator.
Note that $\AAA^*\colon\vv\mapsto (\sum_{\ell=1}^LA_{\ell,b}^*v_{\ell,b})_{1\le b\le B}$
and, from \eqref{e:fff}, \eqref{e:hhh} and 
\cite[Proposition~13.30 and Proposition~16.9]{Bauschke2017},
 $\partial\ff\colon\ww\mapsto\times_{b=1}^B\partial f_{b}(w_b)$ and 
$\partial\hh^*\colon\vv\mapsto\times_{\ell=1}^L\partial 
(h_{\ell}\circ\boldsymbol{\Lambda}_{\ell})^*(\boldsymbol{v}_{\ell})$. Since, for every 
$\ell\in\{1,\ldots,L\}$, 
$h_{\ell}\circ\boldsymbol{\Lambda}_{\ell}$ is convex differentiable with 
a $B\beta_{\ell}-$Lipschitzian gradient
$\nabla(h_{\ell}\circ\boldsymbol{\Lambda}_{\ell})=\boldsymbol{\Lambda}_{\ell}^*
\circ\nabla 
h_{\ell}\circ\boldsymbol{\Lambda}_{\ell}$,
it follows from Baillon-Haddad theorem 
\cite[Corollary~18.17]{Bauschke2017} that 
$\nabla(h_{\ell}\circ\boldsymbol{\Lambda}_{\ell})$ is $(B\beta_{\ell})^{-1}-$cocoercive and, 
hence,
$\partial 
(h_{\ell}\circ\boldsymbol{\Lambda}_{\ell})^*
=\left(\nabla(h_{\ell}\circ\boldsymbol{\Lambda}_{\ell})\right)^{-1}$
is $(B\beta_{\ell})^{-1}-$strongly 
monotone. Therefore, for every $\rho_{\ell}\in[0,(B\beta_{\ell})^{-1}]$,  
$(h_{\ell}\circ\boldsymbol{\Lambda}_{\ell})^*$ is 
$\rho_{\ell}-$strongly convex.
By defining 
\begin{equation}
(\forall \ell \in \{1,\ldots,L\}) \quad \boldsymbol{\varphi}_{\ell}=(h_{\ell}\circ\boldsymbol{\Lambda}_{\ell})^*-\rho_{\ell}\|\cdot\|^2/2,
\end{equation}
it follows from \cite[Proposition~10.8]{Bauschke2017} that, for every $\ell \in \{1,\ldots,L\}$,
$\boldsymbol{\varphi}_{\ell}\in
\Gamma_0(\GG_{\ell}^B)$ and, hence,
$\partial\boldsymbol{\varphi}_{\ell}:=\partial(h_{\ell}\circ\boldsymbol{\Lambda}_{\ell})^*
-\rho_{\ell}\boldsymbol{\Id}$ 
is maximally monotone.
Consequently,
Problem~\ref{problem} can be rewritten equivalently as
\begin{equation}
\label{e:systinclusion}
\text{find}\;\; \quad(\ww,\vv) \in \HHH\times \GGG\quad\text{such that 
}\;\; 
\begin{cases}
(\forall b\in\{1,\ldots,B\})\;\; 0\in\partial f_b(w_b)+B_b(\ww,\vv)\\
(\forall \ell\in\{1,\ldots,L\})\;\; 
0\in\partial\boldsymbol{\varphi}_{\ell}(\vv_{\ell})+B_{\ell}(\ww,\vv),
\end{cases}
\end{equation}
which, for strictly positive constants $(\tau_b)_{1\leq b\leq B}$ and $(\gamma_{\ell})_{1\leq 
\ell\leq L}$, is equivalent to
\begin{equation}
\label{e:systinclusionb}
\text{find}\;\; \quad(\ww,\vv) \in \HHH\times \GGG\quad\text{such that 
}\;\; 
\begin{cases}
(\forall b\in\{1,\ldots,B\})\;\; 0\in\tau_b \partial f_b(w_b)+\tau_b B_b(\ww,\vv)\\
(\forall \ell\in\{1,\ldots,L\})\;\; 
0\in \gamma_\ell\partial\boldsymbol{\varphi}_{\ell}(\vv_{\ell})+\gamma_\ell B_{\ell}(\ww,\vv),
\end{cases}
\end{equation}
where 
\begin{equation}
\begin{cases}
B_b\colon(\ww,\vv)\mapsto\sum_{\ell=1}^LA_{\ell,b}^*v_{\ell,b}\\
B_{\ell}\colon(\ww,\vv)\mapsto-
\left(A_{\ell,1}w_1,\ldots,A_{\ell,B}w_B\right)+\rho_{\ell}\boldsymbol{v}_{\ell}.
\end{cases}
\end{equation}
Since $\boldsymbol{S}\colon 
(\ww,\vv)\mapsto(\AAA^*\vv,-\AAA\ww)$ and $\boldsymbol{D}\colon 
\GGG\to\GGG\colon\boldsymbol{v}\mapsto 
(\rho_{\ell}\boldsymbol{v}_{\ell})_{1\le\ell\le L}$
are linear and monotone operators in $\HHH\times\GGG$ and $\GGG$, respectively,
the operator $$\boldsymbol{B}\colon(\ww,\vv)\mapsto
(\AAA^*\vv,-\AAA\ww+\boldsymbol{D}\boldsymbol{v})=((B_b(\ww,\vv))_{1\le b\le 
	B},(B_{\ell}(\ww,\vv))_{1\le\ell\le 
	L})$$ is maximally monotone
in $\HHH\times\GGG$. Therefore, by defining the strongly positive diagonal linear operator 
\begin{align}
\boldsymbol{U}\colon \HHH\times \GGG &\to \HHH\times \GGG\nonumber\\
(\ww,\vv) &\mapsto (\boldsymbol{T}\ww,\boldsymbol{\Gamma}\vv),
\end{align}
where $\boldsymbol{T}\colon \ww \mapsto (\tau_b w_b)_{1\le b\le B}$ and 
$\boldsymbol{\Gamma}\colon \vv \mapsto (\gamma_\ell \vv_\ell)_{1\le \ell \le L}$,
the operator 
\begin{equation}
\label{e:defUB}
	\boldsymbol{U}\boldsymbol{B}\colon(\ww,\vv)\mapsto
	(\boldsymbol{T}\AAA^*\vv,-\boldsymbol{\Gamma}\AAA\ww+
	\boldsymbol{\Gamma}\boldsymbol{D}\boldsymbol{v})=((\tau_bB_b(\ww,\vv))_{1\le b\le 
		B},(\gamma_{\ell}B_{\ell}(\ww,\vv))_{1\le\ell\le 
		L})
\end{equation}
is maximally monotone
in $(\HHH\times \GGG, \|\cdot\|_{\boldsymbol{U}^{-1}})$, where
\begin{equation}
\big(\forall  (\ww,\vv) \in \HHH\times \GGG\big)\quad
\|  (\ww,\vv) \|_{\boldsymbol{U}^{-1}} = \sqrt{\sum_{b=1}^B \tau_b^{-1} \|w_b\|^2 + \sum_{\ell 
= 1}^L \gamma_\ell^{-1} \|\vv_\ell\|^2}.
\end{equation}
Note that the renormed product space $(\HHH\times \GGG, 
\|\cdot\|_{\boldsymbol{U}^{-1}})$ is the Hilbert sum 
$\HH_1\oplus\cdots\HH_B\oplus\GG^B_1\oplus\cdots\GG^B_L$ where, for every 
$b\in\{1,\ldots,B\}$ and $\ell\in\{1,\ldots,L\}$, $\HH_b$ and $\GG^B_\ell$ are endowed 
by the norm 
$\|\cdot\|_{\tau_b}\colon 
w_b\mapsto\|w_b\|/\sqrt{\tau_b}$
and
$\|\cdot\|_{\gamma_\ell}\colon\boldsymbol{v}_\ell
\mapsto\|\boldsymbol{v}_\ell\|/\sqrt{\gamma_\ell}$, respectively. Therefore,
since $\tau_b\partial f_b$ and $\gamma_{\ell}\partial\boldsymbol{\varphi}_{\ell}$ are 
maximally monotone in $(\HH_b,\|\cdot\|_{\tau_b})$ and 
$(\GG^B_\ell,\|\cdot\|_{\gamma_\ell})$, respectively,
we conclude that \eqref{e:systinclusionb} is a particular case of the primal inclusion in 
\cite[Proposition 5.1]{Combettes2015}.

Now we  write Algorithm~\ref{algo:bcdg} as a particular case of the random 
block-coordinate 
Douglas-Rachford splitting proposed in \cite[Proposition 5.1]{Combettes2015} applied to 
\eqref{e:systinclusionb} in $(\HHH\times \GGG, \|\cdot\|_{\boldsymbol{U}^{-1}})$.
Given 
$(\boldsymbol{t},\boldsymbol{s})\in\HHH\times\GGG$, let 
$(\boldsymbol{w},\boldsymbol{v})=J_{\boldsymbol{U}\boldsymbol{B}}(\boldsymbol{t},
\boldsymbol{s})=(\boldsymbol{\Id}+
\boldsymbol{U}\boldsymbol{B})^{-1}(\boldsymbol{t},\boldsymbol{s})$.
It follows from \eqref{e:defUB} that
\begin{equation}
\label{e:invmatrix}
\begin{cases}
\ww= \tw - \boldsymbol{T} \TT^* \vv\\[0.5em]
\vv= (\boldsymbol{\Id}+ \boldsymbol{\Gamma} \boldsymbol{D})^{-1}(\sv + \boldsymbol{\Gamma} \TT 
\ww),
\end{cases}
\end{equation}
which leads to
\begin{equation}
\label{e:inv2}
\ww = \left(\boldsymbol{\Id} + \boldsymbol{T} \AAA^*  (\boldsymbol{\Id}+ \boldsymbol{\Gamma} 
\boldsymbol{D})^{-1} \boldsymbol{\Gamma}\AAA\right)^{-1}\left(\tw - \boldsymbol{T}\AAA^*  (\boldsymbol{\Id}+ 
\boldsymbol{\Gamma} \boldsymbol{D})^{-1}\sv\right).
\end{equation}
In order to derive an explicit formula for the matrix inversion in \eqref{e:inv2}, 
set $\zz = \tw - \boldsymbol{T}\AAA^*  (\boldsymbol{\Id}+ 
\boldsymbol{\Gamma} \boldsymbol{D})^{-1}\sv$.
We have ${\zz}=\ww+\boldsymbol{T}\AAA^*  (\boldsymbol{\Id}+ 
\boldsymbol{\Gamma} \boldsymbol{D})^{-1}\boldsymbol{\Gamma}\AAA\ww$
and, since \eqref{e:defA} and $\boldsymbol{D}$ is diagonal, we obtain
$$\boldsymbol{T}\AAA^*  (\boldsymbol{\Id}+ 
\boldsymbol{\Gamma} 
\boldsymbol{D})^{-1}\boldsymbol{\Gamma}\AAA\colon\ww\mapsto\left(\tau_b\sum_{\ell=1}^L
\frac{\gamma_\ell A_{\ell,b}^*A_{\ell,b}w_b}{1+\gamma_\ell \rho_{\ell}}\right)_{1\le b\le B}, $$
and, hence,
\begin{equation}
\label{e:13inv}
(\forall b\in\{1,\ldots,B\})\quad  
w_b=\left(\Id+\tau_b\sum_{\ell=1}^L  
\frac{\gamma_\ell A_{\ell,b}^*A_{\ell,b}}{1+\gamma_\ell\rho_{\ell}}\right)^{-1}\!\!\!\!\!z_b
=C_b z_b.
\end{equation}
Therefore, \eqref{e:inv2} can be written equivalently as
\begin{equation}
\label{e:firststep}
(\forall b\in\{1,\ldots,B\})\quad  
w_b=C_b\left(t_b^{[i]}-
\tau_b\sum_{\ell=1}^L\frac{A_{\ell,b}^*s_{\ell,b}}{1+\gamma_\ell\rho_{\ell}}\right)=
C_b(t_b-\tau_b u_b),
\end{equation}
where 
\begin{equation}
(\forall b\in \{1,\ldots,B\}) \qquad u_b= \sum_{\ell=1}^L \frac{A_{\ell,b}^* \, 
s_{\ell,b}}{1+\gamma_\ell\rho_\ell}.
\end{equation}
Moreover, from \eqref{e:invmatrix}, we deduce that
\begin{equation}
(\forall \ell\in\{1,\ldots,L\})\quad \vv_{\ell}=\frac{\sv_{\ell}+
	\gamma_\ell (A_{\ell,b}w_b)_{1\le b\le B}}{1+\gamma_\ell \rho_{\ell}},
\end{equation}
and, hence, we have $J_{\boldsymbol{U}\boldsymbol{B}}\colon(\boldsymbol{t},
\boldsymbol{s})\mapsto\left((Q_b(\boldsymbol{t},
\boldsymbol{s}))_{1\le b\le B},(Q_{\ell}(\boldsymbol{t},
\boldsymbol{s}))_{1\le \ell\le L}\right)$, where 
\begin{equation}
	\begin{cases}
\displaystyle (\forall b\in\{1,\ldots,B\})\quad Q_b\colon (\boldsymbol{t},\sv)\mapsto C_b\left(t_b-
\tau_bu_b\right)\\
\displaystyle (\forall \ell\in\{1,\ldots,L\})\quad Q_{\ell}\colon (\boldsymbol{t},\sv)\mapsto 
\frac{\sv_{\ell}+
	\gamma_\ell (A_{\ell,b}Q_b(\boldsymbol{t},\sv))_{1\le b\le B}}{1+\gamma_\ell\rho_{\ell}}.
\end{cases}
\end{equation}
Now, it follows from \cite[Proposition~16.44]{Bauschke2017} that, for every 
$b\in\{1,\ldots,B\}$ and $\ell\in\{1,\ldots,L\}$, $J_{\tau_b\partial 
f_b}=\prox_{\tau_b f_b}$ and
 $J_{\gamma_\ell\partial\boldsymbol{\varphi}_{\ell}}=\prox_{\gamma_\ell 
 \boldsymbol{\varphi}_{\ell}}$
and,
for every $\ell \in 
	\{1,\ldots,L\}$ 
	and 
$(\rr_\ell,\zz_{\ell})\in 
\GG_\ell^B\times\GG_\ell^B$, we have 
\begin{align}
\rr_\ell = \prox_{\gamma_\ell \boldsymbol{\varphi}_\ell}\zz_\ell
\Leftrightarrow\quad & \frac{\zz_\ell-\rr_\ell }{\gamma_\ell} \in \partial 
\boldsymbol{\varphi}_\ell(\rr_\ell)\nonumber\\
\Leftrightarrow\quad & \frac{\zz_\ell-\rr_\ell }{\gamma_\ell}  \in 
\partial(h_{\ell}\circ\boldsymbol{\Lambda}_{\ell})^*\rr_\ell 
-\rho_{\ell}\rr_\ell \nonumber\\
\Leftrightarrow\quad & \zz_\ell-(1-\gamma_\ell\rho_\ell)\rr_\ell \in \gamma_\ell 
\partial(h_{\ell}\circ\boldsymbol{\Lambda}_{\ell})^*\rr_\ell . \label{eq:prox_htilde_star0}
\end{align}
Therefore, if $\gamma_\ell\rho_\ell  <1$ we have that \eqref{eq:prox_htilde_star0} is 
equivalent to
\begin{equation}
\rr_\ell = \prox_{\frac{\gamma_\ell}{1-\gamma_\ell\rho_\ell} \, 
	(h_\ell\circ\boldsymbol{\Lambda}_{\ell})^*}\Big(\frac{\zz_\ell}{1-\gamma_\ell\rho_\ell} 
\Big)\label{eq:prox_htilde_star}
\end{equation}
and, from Moreau's decomposition formula \cite[Theorem~14.13(ii)]{Bauschke2017}, 
we 
obtain
\begin{equation}
\label{e:calcprox2}
\rr_\ell =  \frac{1}{1-\gamma_\ell\rho_\ell} 
\left(\zz_\ell - \gamma_\ell \prox_{\frac{1-\gamma_\ell\rho_\ell}{\gamma_\ell} 
	(h_\ell\circ\boldsymbol{\Lambda}_{\ell})}\Big(\frac{\zz_\ell}{\gamma_\ell}\Big) 
\right).
\end{equation}
Noting that, for every $\ell \in \{1,\ldots,L\}$, 
$\boldsymbol{\Lambda}_{\ell}\circ\boldsymbol{\Lambda}_{\ell}^*=B\Id$,
from \cite[Proposition~24.14]{Bauschke2017}, \eqref{eq:prox_htilde_star0} and 
\eqref{e:calcprox2} we deduce that
\begin{equation}\label{e:proxhtilde}
(\forall b\in\{1,\ldots,B\})\quad r_{\ell,b}=
\frac{1}{B(1-\gamma_\ell\rho_\ell)} \left(\sum_{d=1}^B 
z_{\ell,d}-\gamma_\ell\prox_{\frac{B(1-\gamma_\ell\rho_\ell)}{\gamma_\ell} 
	h_\ell}\left(\frac{1}{\gamma_\ell}\sum_{d=1}^B 
z_{\ell,d}\right)\right)
\end{equation}
and, hence,
\begin{equation}
\prox_{\gamma_\ell \boldsymbol{\varphi}_\ell}\zz_\ell=\frac{1}{B(1-\gamma_\ell\rho_\ell)} 
\left(\sum_{d=1}^B 
z_{\ell,d}-\gamma_\ell\prox_{\frac{B(1-\gamma_\ell\rho_\ell)}{\gamma_\ell} 
	h_\ell}\left(\frac{1}{\gamma_\ell}\sum_{d=1}^B 
z_{\ell,d}\right)\right)_{1\le b\le B}.
\end{equation}
Therefore, by defining, for every $i\in\NN$ and $\ell\in\{1,\ldots,L\}$, 
${\boldsymbol{e}}_{\ell}^{[i]}\in\GG_{\ell}^B$ 
via
\begin{equation}
\label{e:deftilded2}
(\forall \ell \in \{1,\ldots,L\})\quad  \boldsymbol{e}_{\ell}^{[i]} =  
\left(-\frac{\gamma_\ell}{B(1-\gamma_\ell \rho_\ell)} d_\ell^{[i]}\right)_{1\le b\le B},
\end{equation}
we deduce that Algorithm~\ref{algo:bcdg} can be written equivalently as

\newpage
\noindent \textrm{For }\; $i = 0, 1, \dots$\\[-1em]
\[\hspace{-2.3cm}
\left\lfloor
\begin{aligned}
&\textrm{For $b =1,\ldots,B$}\\[-.5em]
&\left\lfloor
\begin{aligned}
&w_b^{[i+1]}=w_b^{[i]}+\varepsilon_b^{[i]}\left(Q_b(\boldsymbol{t}^{[i]},\sv^{[i]})-w_b^{[i]}\right)\\
&t_b^{[i+1]}=t_b^{[i]} + \varepsilon_b^{[i]}\mu^{[i]} \left( \prox_{\tau_b 
	f_b}(2w_b^{[i+1]}-t_b^{[i]}) + 
a_b^{[i]} - w_b^{[i+1]} \right)\\
\end{aligned}
\right.\\
&\textrm{For $\ell =1,\ldots,L$}\\[-.5em]
&\left\lfloor
\begin{aligned}
&\vv_{\ell}^{[i+1]}=\vv_{\ell}^{[i]}+\varepsilon_{B+\ell}^{[i]}\left(Q_{\ell}(\boldsymbol{t}^{[i]},\sv^{[i]})-\vv_{\ell}^{[i]}\right)\\
&\sv_{\ell}^{[i+1]}=\sv_{\ell}^{[i]} + \varepsilon_{B+\ell}^{[i]}\mu^{[i]} \left( \prox_{\gamma_\ell 
	\boldsymbol{\varphi}_{\ell}}(2\vv_{\ell}^{[i+1]}-\sv_{\ell}^{[i]}) + 
\boldsymbol{e}_{\ell}^{[i]} - \vv_{\ell}^{[i+1]} \right).
\end{aligned}
\right.
\end{aligned}
\right.
\]
Defining, for every $i\in\NN$, $\boldsymbol{a}^{[i]}=\big((a_b^{[i]})_{1\le b\le 
B},(\boldsymbol{e}_{\ell}^{[i]})_{1\le \ell\le L}\big)\in \HHH\times\GGG$, we have 
\begin{align}
\sum_{i\in\NN}\sqrt{\EE(\|\boldsymbol{a}^{[i]}\|^2_{\boldsymbol{U}^{-1}}|\boldsymbol{\chi}^{[i]})} &=
\sum_{i\in\NN}\sqrt{\sum_{b=1}^B\tau_b^{-1}\EE\big(\|{a}_b^{[i]}\|^2|\boldsymbol{\chi}^{[i]}\big)+
	\sum_{\ell=1}^L\gamma_\ell^{-1} 
	\EE\big(\|\boldsymbol{e}_{\ell}^{[i]}\|^2|\boldsymbol{\chi}^{[i]}\big)}\nonumber\\
&\le \sum_{b=1}^B \tau_b^{-1/2} 
\sum_{i\in\NN}\sqrt{\EE(\|{a}_b^{[i]}\|^2|\boldsymbol{\chi}^{[i]})}+
	\sum_{\ell=1}^L{\gamma_\ell}^{-1/2}\sum_{i\in\NN}\sqrt{\EE(\|\boldsymbol{e}_{\ell}^{[i]}\|^2|\boldsymbol{\chi}^{[i]})}
	\nonumber\\
&<+\infty,
\end{align}
where the last inequality follows from \ref{a:prop1i}, \ref{a:prop1ii},
\eqref{e:deftilded2} and
\begin{equation}
\label{e:errorb}
\sum_{i\in\NN}\sqrt{\EE(\|\boldsymbol{e}_{\ell}^{[i]}\|^2|\boldsymbol{\chi}^{[i]})}
=\frac{\gamma_\ell}{\sqrt{B}(1-\gamma_\ell\rho_{\ell})}\sum_{i\in\NN}
\sqrt{\EE(\|d_{\ell}^{[i]}\|^2|\boldsymbol{\chi}^{[i]})}<+\infty.
\end{equation}
Altogether, since operator $J_{\boldsymbol{U}\boldsymbol{B}}$ is weakly 
sequentially continuous because it is continuous and linear, the result follows from
\cite[Proposition 5.1]{Combettes2015} when the error term in the computation of
$J_{\boldsymbol{U}\boldsymbol{B}}$ 
is zero.
\qed\end{proof}
\begin{remark}\
\begin{enumerate}
%
\item  In Proposition~\ref{prop:algo}, 
the binary variables $\varepsilon_b^{[i]}$ and $\varepsilon_{B+\ell}^{[i]}$ signal whether the variables $t_b^{[i]}$ and $\sv_\ell^{[i]}$ are activated or not at iteration $i$. Assumption~\ref{prop:algo:ass_prob} guarantees that each of the latter variables is activated with a nonzero probability at each iteration.
In particular, it must be pointed out that the variables  $p_\ell^{[i]}$ and $q_\ell^{[i]}$ only need to be computed when  $\varepsilon_{B+\ell}^{[i]}=1$.
\item Note that Algorithm \ref{algo:bcdg} may look similar to the stochastic approach proposed 
in  \cite[Corollary 5.5]{Combettes2015} (see also \cite[Remark~2.9]{Briceno2011}, and 
\cite{Bot2013} for deterministic variants). It exhibits however three key differences. Most 
importantly, the operator inversions performed at the initial step amount to inverting a set of 
positive definite self-adjoint operators defined on the spaces $(\HH_b)_{1\le b \le B}$. We 
will see in our application that this reduces to invert a set of small size symmetric  positive 
definite matrices. Another advantage is that the smoothness of the functions $(h_\ell)_{1 
	\leq \ell \leq L}$ is taken into account, and a last one is that the dual variables  appear 
explicitly in the iterations.
\item If, for every $\ell \in \{1,\ldots,L\}$, $\rho_\ell = 0$ and $B = 1$, Algorithm \ref{algo:bcdg} simplifies to Algorithm~\ref{algo:bcdgs}, where unnecessary indices have been dropped and we have set
\begin{equation}
(\forall i \in \NN)\quad 
\begin{cases}
\widetilde{u}^{[i]} = -\tau u^{[i]}\\
(\forall \ell \in \{1,\ldots,L\}\quad \widetilde{s}_{\ell}^{[i]}  = -\tau s_{\ell}^{[i]}.
\end{cases}
\end{equation}
In this case,
\begin{equation}
C = \Big(\Id+\tau \sum_{\ell=1}^L \gamma_\ell A_{\ell}^* \, A_{\ell} \Big)^{-1}.
\end{equation}

\begin{algorithm}
\caption{Random Douglas-Rachford splitting for solving Problem \ref{problem} when $\rho_\ell = 0$ and $B=1$}\label{algo:bcdgs}
{\small
\textsc{Initialization}\\[-0.5em]
\[\hspace{-6.5cm}
\left\lfloor
\begin{aligned}
&\textrm{Set $\tau\in ]0,+\infty[$ and $\eta \in 
]0,1]$.}\\
&\textrm{For every $\ell \in \{1,\ldots,L\}$, set $\gamma_\ell > 0$.}\\
&\textrm{$\widetilde{u}^{[0]} = \sum_{\ell=1}^L  A_{\ell}^* \, \widetilde{s}_{\ell}^{[0]}$}\\[0.25em]
\end{aligned}
\right.
\]
\noindent \textsc{For}\; $i = 0, 1, \dots$\\[-0.5em]	
\[
\!\!\!\!\!\!\!\!\!\!\!\!\!\!\!\!\!\!\!\!\!\!\!\!\!\!\!\!\!\!\!\!\!\!\!\!\!\!\!\!\!\!\!\!\!\!\!\!\!\!\!\!\!\!\!\!\!\!\!\!\!\!\!\!\!\!\!\!\!\!\!
\left\lfloor
\begin{aligned}
&\textrm{Set}\; \mu^{[i]}\in\left]\eta,2-\eta\right[\\
%
& w^{[i+1]} = w^{[i]}+ \varepsilon^{[i]}\left(C\left(t^{[i]}+\widetilde{u}^{[i]}\right)-w^{[i]}\right) \\
& t^{[i+1]} = t^{[i]} + \varepsilon^{[i]}\mu^{[i]} \left( \prox_{\tau 
f}(2w^{[i+1]}-t^{[i]}) + 
a^{[i]} - w^{[i+1]} \right)\\[0.5em]
&\textrm{for $\ell=1,\dots,L$}\\
&\left\lfloor
\begin{aligned}
&q_\ell^{[i]} = \prox_{\frac{h_\ell}{\gamma_\ell} }\left(2A_{\ell} 
w^{[i]}-\frac{\widetilde{s}_{\ell}^{[i]}}{\tau\gamma_\ell}\right) + d_\ell^{[i]}\\
&\widetilde{s}_{\ell}^{[i+1]} = \widetilde{s}_{\ell}^{[i]} + \varepsilon_{\ell+1}^{[i]}\mu^{[i]}\tau\gamma_\ell \Big(q_\ell^{[i]}- A_{\ell} 
w^{[i]}\Big)
\end{aligned}
\right.\\[0.5em]
&\begin{aligned}
\quad &
\widetilde{u}^{[i+1]} = \widetilde{u}^{[i]} + \sum_{\ell=1}^L   \varepsilon_{\ell+1}^{[i]} A_{\ell}^* 
\,  \big(\widetilde{s}_{\ell}^{[i+1]}-\widetilde{s}_{\ell}^{[i]}\big).\\
\end{aligned}\\[-0.5em]
\end{aligned}
\right.
\]
}
\end{algorithm}

When $(\forall \ell \in \{1,\ldots,L\})$ $\gamma_\ell = 1/\tau$, it turns out  this algorithm is exactly the same as the one resulting from a direct application of \cite[Corollary 5.5]{Combettes2015}\cite{Chierchia2017}.

\item The situation when, for a given $\ell$, $h_\ell \in \Gamma_0(\GG_\ell)$ is not Lipschitz-differentiable can be seen as the limit case when $\beta_\ell \to +\infty$. It can then be shown that Algorithm \ref{algo:bcdg} remains valid by setting $\rho_\ell = 0$. 
\end{enumerate}
\end{remark}

\section{Sparse logistic regression}\label{sec:logistic}
The proposed algorithm can be applied in the context of binary linear classification. 
%
%
%
%
A binary linear classifier can be modeled as a function that predicts the output $y \in \{-1,+1\}$ associated to a given input $\xx \in \RR^N$. This prediction is defined through a linear combination of the input components, yielding the decision variable
\begin{equation}\label{eq:prediction}
	d_{\ww}(x) = \sign\big( \xx^\top \ww \big),
\end{equation}
where $\ww\in\RR^N$ is the weight vector to be estimated.
In supervised learning, this weight vector is determined from a set of input-output pairs
\begin{equation}
	\mathcal{S} = \big\{ (\xx_\ell,y_\ell) \in \RR^N \times \{-1,+1\} \;|\; \ell\in\{1,\dots,L\}\big\},
\end{equation}
which is called \emph{training set}. More precisely, the learning task can be defined as the trade-off between fitting the training data and reducing the model complexity, leading to an optimization problem expressed as
\begin{equation}\label{eq:learning_task}
	\minimize_{\ww \in \RR^N} \ff(\ww) + \sum_{\ell=1}^L h\left(y_\ell\,\xx_\ell^\top \ww\right),
\end{equation}
where $\ff\in\Gamma_0(\RR^N)$ is a regularization function and
$h\in\Gamma_0(\RR)$ stands for the loss function. In the context of sparse learning, a 
popular choice for the regularization is the $\ell_1$-norm.
Although many choices for the loss function are possible, we are primarily interested in the logistic loss, which is detailed in the next section.

\subsection{Logistic regression}\label{subsec:logit}
Logistic regression aims at maximizing the posterior probability density function of the weights given the training data, here assumed to 
be a realization of statistically independent input-output random variables. This leads us to
\begin{equation}\label{eq:map}
	\maximize_{\ww\in\RR^N} 
	\boldsymbol{\varphi}(\ww) \prod_{\ell=1}^L \pi(y_\ell \mid \xx_\ell,\ww)\theta_\ell(\xx_\ell|\ww),
\end{equation}
where $\boldsymbol{\varphi}$  is the weight prior probability density function and, for every $\ell \in \{1,\ldots,L\}$, 
$\theta_\ell$ is the conditional data likelihood of the $\ell$-th input knowing the weight values,
while $\pi(y_\ell \mid \xx_\ell,\ww)$ is the conditional probability of the $\ell$-th output knowing the $\ell$-th input and the weights.
Let us model this conditional probability with the sigmoid function defined as
\begin{equation}\label{eq:likelihood}
	\pi(y_\ell\mid \xx_\ell,\ww) = \frac{1}{1+\exp(-y_\ell \xx_\ell^\top \ww)}, 
\end{equation}
and assume that the inputs and the weights are statistically independent and
that $\varphi$ is log-concave. 
Then, the negative-logarithm of the energy in \eqref{eq:map} yields an instance of Problem~\eqref{eq:learning_task} in which
\begin{equation}\label{eq:logistic}
	(\forall v\in\RR)\qquad h(v) = \log\big(1+\exp(-v)\big)
\end{equation}
and, for every $\ww\in \RR^N$, $\ff(\ww) = -\log \boldsymbol{\varphi}(\ww)$. 
(The term $\prod_{\ell=1}^L \theta_\ell(\xx_\ell|\ww)$ can be discarded since  the inputs and the weights are assumed statistically independent.)
The function in \eqref{eq:logistic} is called \emph{logistic loss}. For completeness, note 
that other loss functions, leading to different kinds of classifiers, are the \emph{hinge loss} \cite{Cortes1995}
\begin{equation}
(\forall v\in\RR)\qquad h^{\textrm{hinge}}(v) = \big(\max\{0,1-v\}\big)^{\sf q}
\end{equation}
with ${\sf q}\in\{1,2\}$, and the \emph{Huber loss} \cite{Martins2016}
\begin{equation}
(\forall v\in\RR)\qquad h^{\textrm{huber}}(v) = 
\begin{cases}
0 &\textrm{if $v\ge1$}\\
-v &\textrm{if $v\le-1$}\\
\frac{1}{4}(v-1)^2 &\textrm{otherwise}.
\end{cases}
\end{equation}
These functions can be also handled by the proposed algorithm.

\subsection{Optimization algorithm}\label{subsec:optim}
Let us blockwise decompose the weight variable $\ww\in\RR^N$ as
\begin{equation}\label{eq:block_decomp}
\ww^{\top}=\left[ w_1^{\top} \;\ldots\; w_B^{\top} \right],
\end{equation}
where, for every $b\in \{1,\ldots,B\}$,  $w_b\in\RR^{N_b}$ and
$N_1,\ldots,N_B$ are strictly positive integers such that
$N=N_1+\dots+N_B$. Let us also decompose the input vector 
as $\xx^{\top}=\left[ x_1^{\top} 
\;\ldots\; x_B^{\top} \right]$. Finally, let us assume that the regularization function
is block-separable, i.e. $\boldsymbol{f}=\oplus_{b=1}^Bf_b$,
where, for every $b\in \{1,\ldots,B\}$, $f_b \in \Gamma_0(\RR^{N_b})$.
A typical example of such functions is given by
\begin{equation}\label{e:exfb}
(\forall b \in \{1,\ldots,B\})
\qquad f_b = \lambda\|\cdot\|_{\kappa_b},
\end{equation}
where $\lambda \in [0,+\infty[$ and $\| \cdot \|_{\kappa_b}$, $\kappa_b\in [1,+\infty]$,
denotes the $\ell_{\kappa_b}$-norm of $\RR^{N_b}$.
In particular, when for every $b\in \{1,\ldots,B\}$ $\kappa_b = 1$, $\ff$ reduces to the standard
$\ell_1$ regularizer, whereas setting $\kappa_b \equiv 2$
results in a potential promoting group sparsity \cite{Bach2012}.

In the context described above, \eqref{eq:learning_task} is a particular case of 
Problem~\ref{problem}
where, for every $b \in \{1,\ldots,B\}$,  $\HH_b = \RR^{N_b}$, 
for every $\ell \in \{1,\ldots,L\}$,
$\GG_\ell = \RR$, $h_{\ell}=h$ and $A_{\ell,b}=y_{\ell}x_{\ell,b}^{\top}$.
 Note that since, for every $\ell\in\{1,\ldots,L\}$, $y_{\ell}^2=1$, $A_{\ell,b}^*A_{\ell,b} =  
 x_{\ell,b}x_{\ell,b}^\top$. Moreover,
$h$ defined in \eqref{eq:logistic} is twice differentiable with
\begin{align}
(\forall v \in \RR)\qquad
h'(v)&=-\frac{\exp(-v)}{1+\exp(-v)}, \label{e:derlog}\\
h''(v)&=\frac{\exp(-v)}{(1+\exp(-v))^2}.
\end{align}
 Since $h''$ is maximized in $v=0$, we have
$\sup_{v\in\RR}|h''(v)|=1/4$, which implies that $h'$ is $1/4$-Lipschitz continuous
and we have thus, for every $\ell\in\{1,\ldots,L\}$, $\beta_\ell = 1/4$.

The problem can thus be solved with Algorithm \ref{algo:bcdg}, the convergence of which 
is guaranteed almost surely under the assumptions of Proposition~\ref{prop:algo}.

\subsection{Proximity operator}\label{subsec:prox}
An efficient computation of the proximity operators of functions $(f_b)_{1\le b \le B}$ and $h$ plays a crucial role in the implementation of Algorithm~\ref{algo:bcdg}. 
There exists an extensive literature on the computation of the proximity operators of functions like \eqref{e:exfb} \cite{ProxRepository}.
In particular,
when $\kappa_b = 1$ (resp. $\kappa_b = 2$), this proximity operator reduces to a 
component-wise (resp. blockwise) soft-thresholding \cite{Combettes2010}.
Regarding the logistic loss in \eqref{eq:logistic}, although some numerical methods exist \cite{Parikh2014,Wang2016}, to the best of our knowledge, 
no thorough investigation of the form of its proximity operator has been made. 
The next proposition will contribute to fill such a void.
The result relies on the generalized $\operatorname{W}$-Lambert function  recently 
introduced in \cite{Mezo2014,Maignan2016}, defined via 
\begin{equation}
(\forall \bar{v}\in\RR)(\forall v\in\RR)(\forall r\in\left]0,+\infty\right[)\quad \bar{v}\,
(\exp(\bar{v}) + r)= v \quad\Leftrightarrow\quad \bar{v} = \operatorname{W}_r(v). 
\label{eq:wlambertprop}
\end{equation} 
When $r\in [\exp(-2),+\infty[$, $\operatorname{W}_r$ is uniquely defined and strictly increasing, but when $r \in ]0,\exp(-2)[$, there exist three branches for $\operatorname{W}_r$.
We will retain the only one which can take nonnegative values  (denoted by $\operatorname{W}_{r,0}$ in \cite[Theorem 4]{Mezo2014}) and is also strictly increasing.
This function can be efficiently evaluated through a Newton-based method devised by Mez{\H{o}} \emph{et al.} \cite{Mezo2014} and available on line.\footnote{https://sites.google.com/site/istvanmezo81/others}

\begin{proposition}
	\label{prop:proxh}
Let $\gamma\in\left]0,+\infty\right[$ and let $h\colon v\mapsto \log\big(1+\exp(-v)\big)$. We have
\begin{equation}\label{eq:prox}
(\forall v\in\RR)\quad \prox_{\gamma h} (v) = v + \operatorname{W}_{\exp(-v)} \left(\gamma \exp \left(-v \right) \right).
\end{equation}
\end{proposition}
\begin{proof}
Let $v\in\RR$ and $\gamma\in\left]0,+\infty\right[$. For every $p\in\RR$, it follows from the 
definition of $\prox_{\gamma h}$, \eqref {e:derlog} and \eqref{eq:wlambertprop} that
\begin{align}
p=\prox_{\gamma h}(v)\quad
&\Leftrightarrow\quad  
v-p=-\frac{\gamma\exp(-p)}{1+\exp(-p)}=-\frac{\gamma}{\exp(p)+1}\label{e:aux}\\
&\Leftrightarrow\quad  
(p-v)(\exp(p)+1)=\gamma\nonumber\\
&\Leftrightarrow\quad  
(p-v)(\exp(p-v)+\exp(-v))=\gamma\exp(-v)\nonumber\\
&\Leftrightarrow\quad  
p-v=W_{\exp(-v)}(\gamma\exp(-v))
\end{align}
and the result follows.
\qed
\end{proof}

From a numerical standpoint, it must be emphasized that the exponentiation in \eqref{eq:prox} may be problematic, as it yields an arithmetic overflow when $v$ tends to $-\infty$. To overcome this issue, one can use the asymptotic equivalence\footnote{Hereafter, following Landau's notation,  we will write that $F(v) = \vartheta(G(v))$,
where $F\colon \RR \to \RR$ and $G\colon \RR \to \RR$, if $F(v)/G(v) \to 0$
as $v \to +\infty$ (or $v \to -\infty$).}
between the proximity operator of the logistic function and other more tractable functions.
\begin{proposition}
Let $\gamma\in\left]0,+\infty\right[$ and let $h\colon v\mapsto \log(1+\exp(-v))$. Then, as 
$v \to -\infty$,
\begin{equation}
\prox_{\gamma h}(v) 
= v + \gamma \big(1 - \exp(\gamma + v) + (1+\gamma)  \exp( 2(\gamma + v))\big) 
+ \vartheta(\exp(2v)).
\end{equation} 
\end{proposition}
\begin{proof}
Define
\begin{equation}\label{e:deft}
(\forall v\in\RR) \quad \varphi(v) = W_{\exp(-v)}(\gamma\exp(-v)).
\end{equation}
According to Proposition~\ref{prop:proxh}, 
\begin{equation}\label{e:phiproxh}
\varphi=\prox_{\gamma h}-\Id.
\end{equation}
It follows from  \eqref{e:aux} that 
${\rm ran }\,\varphi={\rm ran}(\prox_{\gamma h}-\Id)\subset \left]0,\gamma\right[$.
Moreover, from \cite[Section~24.2]{Bauschke2017} we deduce that $\varphi=-\gamma\prox_{h^*/\gamma}(\cdot/\gamma)$, 
which is decreasing and continuous by virtue of
\cite[Proposition~24.31]{Bauschke2017}. Therefore, $\lim_{v\to-\infty}\varphi(v)$ exists and
from \eqref{eq:wlambertprop} we have
\begin{equation}
\label{e:interm}
(\forall v\in\RR) \quad  \varphi(v)\exp(\varphi(v))=(\gamma-\varphi(v))\exp(-v).
\end{equation}
Since the left side of the equality above is bounded, we deduce that 
$\lim_{v\to-\infty}\varphi(v)=\gamma$. Subsequently, we define
\begin{equation}
\label{e:defu}
(\forall v\in\RR)\quad u(v)=\varphi(v)-\gamma\quad \text{satisfying}\quad 
\lim_{v\to-\infty}u(v)=0.
\end{equation}
Hence, \eqref{e:interm} can be rewritten as
\begin{equation}
(\gamma+u(v))\exp(\gamma)\exp(u(v))=-u(v)\exp(-v)
\end{equation}
and by using the first order Taylor expansion around $\xi=0$, $\exp(\xi)=1+\xi+\vartheta(\xi)$ 
and the fact that $u(v)^2+(\gamma+u(v))\vartheta(u(v))=\vartheta(u(v))$,
we obtain
\begin{align}
\label{e:interm2}
u(v)&=-\exp(\gamma+v)(\gamma+u(v))(1+u(v)+\vartheta(u(v)))\nonumber\\
&=-\exp(\gamma+v)(\gamma+(\gamma+1)u(v)+\vartheta(u(v)))\nonumber\\
&=-\gamma\exp(\gamma+v)-(\gamma+1)\exp(\gamma+v)u(v)-\exp(\gamma+v)\vartheta(u(v)).
\end{align}
We deduce from this relation that
\begin{equation}\label{e:interm2.5}
u(v)=-\frac{\gamma\exp(\gamma+v)+\exp(\gamma+v)\vartheta(u(v))}{1+(\gamma+1)\exp(\gamma+v)}.
\end{equation}
It follows that
\begin{equation}\label{e:interm3}
\lim_{v \to -\infty} u(v)\exp(-v) =- \gamma \exp(\gamma),
\end{equation}
which implies that $\exp(\gamma+v)\vartheta(u(v))=\vartheta(\exp(2v))$
and, from \eqref{e:interm2.5} we obtain
\begin{equation}\label{e:uord2}
u(v)=-\frac{\gamma\exp(\gamma+v)}{1+(\gamma+1)\exp(\gamma+v)}+\vartheta(\exp(2v)).
\end{equation}
Combining \eqref{e:phiproxh}, \eqref{e:defu}, and \eqref{e:uord2} yields
\begin{align}
\prox_{\gamma h}(v)&=v+\gamma\left(1-\frac{\exp(\gamma+v)}
{1+(\gamma+1)\exp(\gamma+v)}\right)+\vartheta(\exp(2v))\nonumber\\
&=v+\gamma\left(\frac{1+\gamma\exp(\gamma+v)}
{1+(\gamma+1)\exp(\gamma+v)}\right)+\vartheta(\exp(2v))\nonumber\\
&=v+\gamma\left(1-\exp(\gamma+v)+(1+\gamma)\exp(2(\gamma+v))\right)
+\vartheta(\exp(2v)),
\end{align}
where the last equality follows from the second order Taylor expansion around $\xi = 0$.
\qed
\end{proof}

\section{Experimental results}\label{sec:four}
In order to assess the performance of Algorithm~\ref{algo:bcdg}, we performed the training on standard datasets\footnote{http://www.csie.ntu.edu.tw/$\sim$cjlin/libsvmtools/datasets/binary.html}\footnote{http://yann.lecun.com/exdb/mnist} (see Table~\ref{tab:datasets}), and we compared it with the following approaches.
\begin{itemize}
\item Stochastic Forward-Backward splitting (SFB) \cite{Richtarik2014, Rosasco2014, Combettes2016, Atchade2017}
\begin{equation*}
\begin{aligned}
&w^{[0]}\in\RR^N\\
&\textrm{For $i=0,1,\dots$}\\
&\left\lfloor
\begin{aligned}
&\textrm{Select $\;\LL^{[i]} \subset \{1,\dots,L\}$}\\
%
%
& w^{[i+1]} = \prox_{\gamma_i f}\Big( w^{[i]} - \gamma_i \sum_{\ell\in\LL^{[i]}} y_\ell x_\ell h'\big(y_\ell x_\ell^\top w^{[i]}\big) \Big)
\end{aligned}
\right.
\end{aligned}
\end{equation*}
where $(\gamma_i)_{i\in\NN}$ is a decreasing sequence of positive values.

\item 
Regularized Dual Averaging (RDA) \cite{Xiao2010}
\begin{equation*}
\begin{aligned}
&w^{[0]}\in\RR^N,\; z^{[0]} = 0\\
&\textrm{For $i=0,1,\dots$}\\
&\left\lfloor
\begin{aligned}
&\textrm{Select $\;\LL^{[i]} \subset \{1,\dots,L\}$}\\
& z^{[i+1]} = z^{[i]} + \sum_{\ell\in\LL^{[i]}} y_\ell x_\ell h'\big(y_\ell x_\ell^\top w^{[i]}\big)\\
%
%
& w^{[i+1]} = \prox_{\gamma_i f}\big( - \gamma_i \, z^{[i+1]} \big)
\end{aligned}
\right.
\end{aligned}
\end{equation*}
where $(\gamma_i)_{i\in\NN}$ is a decreasing sequence of positive values.

\item Block-Coordinate Primal-Dual splitting (BCPD) \cite{Chierchia2016}
\begin{equation*}
\begin{aligned}
&w^{[0]} \in \RR^N,\; v^{[0]} \in \RR^L\\
&\textrm{For $i=0,1,\dots$}\\
&\left\lfloor
\begin{aligned}
&\textrm{Select $\;\LL^{[i]} \subset \{1,\dots,L\}$}\\
&w^{[i+1]} = \prox_{\tau f}\big(w^{[i]} - \tau u^{[i]}\big)\\
& (\forall \ell\in\LL^{[i]})\quad v^{[i+1]}_\ell = \prox_{\sigma h^*}\Big(v^{[i]}_\ell + \sigma y_\ell x_\ell^\top \big(2w^{[i+1]} - w^{[i]}\big)\Big)\\
& (\forall \ell\notin\LL_i)\quad v^{[i+1]}_\ell = v^{[i]}_\ell\\[0.5em]
&u^{[i+1]} = u^{[i]} + \sum_{\ell\in\LL^{[i]}} \big(v_\ell^{[i+1]}-v_\ell^{[i]}\big)y_\ell x_\ell \\
\end{aligned}
\right.
\end{aligned}
\end{equation*}
where $\tau>0$ and $\sigma>0$ are such that $\tau\sigma\big\|\sum_{\ell=1}^L x_\ell x_\ell^\top\big\| \le 1$.
\end{itemize}

The algorithmic parameters are reported in Table~\ref{tab:params}. For all the algorithms, mini-batches of size $1000$ were randomly selected using a uniform distribution, and the initial vector $w^{[0]}$ was randomly drawn from the normal distribution with zero mean and unit variance. For the datasets with more than two classes (MNIST and RCV1), the ``one-versus-all'' approach is used \cite{Rifkin2004}. All experiments were carried out with Matlab 2015a on an Intel i7 CPU at 3.40 GHz and 12 GB of RAM.

\begin{table}
	\centering
	\caption{Training sets used in the experiments ($K$ is the number of classes).}
	\label{tab:datasets}
	\begin{tabular}{lcccl}
		\toprule
		Dataset & $N$ & $L$ & $K$ & \\
		\midrule
		W8A   & 300   & 49749 & 2  \\
		MNIST & 717   & 60000 & 10 \\
		RCV1  & 12560 & 30879 & 20 \\
		\bottomrule
	\end{tabular}
\end{table}

\begin{table}
\centering
\caption{Algorithmic parameters used in the experiments.}
\label{tab:params}
\begin{tabular}{lccccccc}
\toprule
& SFB / RDA & \multicolumn{4}{c}{Algo~\ref{algo:bcdg}} & \multicolumn{2}{c}{BCPD} \\
\cmidrule(lr){2-2} \cmidrule(lr){3-6} \cmidrule(lr){7-8} 
Dataset & $\gamma_i$ & $\gamma,\tau$ & $\mu^{[i]}$ & $\rho$ & $B$ & $\tau$ & $\sigma$\\
\midrule
W8A     &  $10^{-1}/\sqrt{i+1}$ & \multirow{3}{*}{$1$} & \multirow{3}{*}{$1.5$} & \multirow{3}{*}{$0.1$} & 1 & \multirow{3}{*}{$0.1$} & \multirow{3}{*}{$\displaystyle\tau^{-1}\big\|\sum_{\ell=1}^L x_\ell x_\ell^\top\big\|^{-1}$}\\[0.25em]
MNIST   & $1/\sqrt{i+1}$ & &&& 1\\[0.25em]
RCV1  & $10/\sqrt{i+1}$ & &&& 9\\
\bottomrule
\end{tabular}
\end{table}

Table~\ref{tab:test} reports the classification performance achieved by the compared algorithms, which includes the classification errors computed on a (disjoint) test set, as well as the sparsity degree of the solution. For all the considered datasets, the regularization parameter $\lambda$ was selected with a cross-validation procedure. The results show that the proposed algorithm finds a solution that yields the same accuracy as state-of-the-art methods, while being sparser than the ones produced by gradient-like methods (SFB and RDA). 

Figure~\ref{fig:comparison} reports the training performance versus time of the considered algorithms, which includes the criterion in \eqref{eq:learning_task}, and the distance to the solution $w^{[\infty]}$ obtained after many iterations for each compared method. The results indicate that the proposed approach converges faster to a smaller value of the objective criterion. This could be related to the implicit preconditioning present in Algorithm~\ref{algo:bcdg} through the matrix $Q$. Another interesting feature of our algorithm is the free choice of parameters $\gamma$ and $\mu_i$. Conversely, in both SFB and RDA, the parameter $\gamma_i$ (also referred to as \emph{learning rate}) needs to be carefully selected by hand, causing such algorithms to slow down or even diverge if the learning rate is chosen too small or too high. 

\begin{table}
	\centering
	\caption{Classification performance on test sets (after training for a fixed number of iterations).}
	\label{tab:test}
	{\scriptsize
	\begin{tabular}{@{}lc@{ -- }cc@{ -- }cc@{ -- }cc@{ -- }c@{}}
		\toprule
		\multirow{2.5}*{\sc Dataset} & \multicolumn{2}{c}{Algo~\ref{algo:bcdg}} &  \multicolumn{2}{c}{SFB} &  \multicolumn{2}{c}{RDA} &  \multicolumn{2}{c}{BCPD} \\
		\cmidrule(lr){2-3} \cmidrule(lr){4-5} \cmidrule(lr){6-7} \cmidrule(lr){8-9} 
		 & Errors & Zeros & Errors & Zeros & Errors & Zeros & Errors & Zeros\\
		\midrule
		\textsc{w8a}     & 9.73\% & 19.60\% & 9.92\% & 5.65\% & 9.99\% & 0.33\% & 10.44\% & 50.50\% \\
		\textsc{mnist}   & 8.49\% & 41.57\% & 8.37\% & 5.25\% & 8.60\% & 11.13\% & 8.45\% & 58.64\%\\
		\textsc{rcv1}    & 6.62\% & 83.43\% & 6.67\% & 35.22\% & 6.60\% & 32.90\% & 6.25\% & 98.57\% \\
		\bottomrule
	\end{tabular}
}
\end{table}

\begin{figure}
\centering
\subfloat[W8A: criterion in \eqref{eq:learning_task} vs time]{\includegraphics[width=0.45\linewidth]{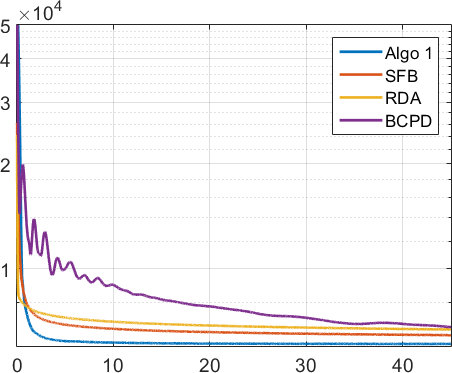}}
\hfill
\subfloat[W8A: distance to $w^{[\infty]}$ vs time]{\includegraphics[width=0.45\linewidth]{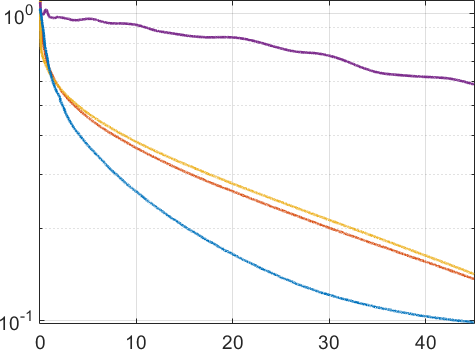}}
\\
\subfloat[MNIST: criterion in \eqref{eq:learning_task} vs time]{\includegraphics[width=0.45\linewidth]{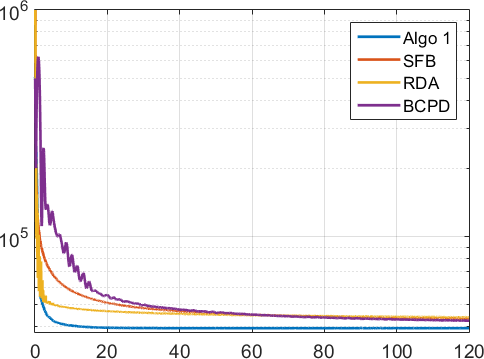}}
\hfill
\subfloat[MNIST: distance to $w^{[\infty]}$ vs time]{\includegraphics[width=0.45\linewidth]{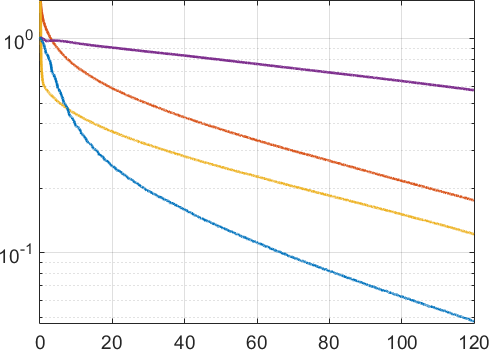}}
\\
\subfloat[RCV1: criterion in \eqref{eq:learning_task} vs time]{\includegraphics[width=0.45\linewidth]{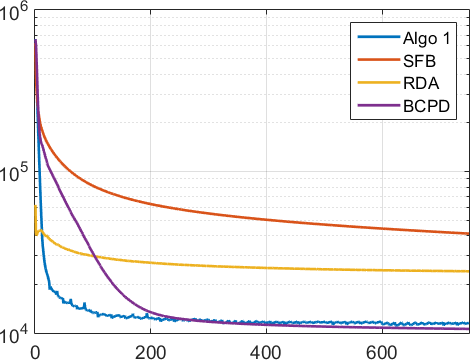}}
\hfill
\subfloat[RCV1: distance to $w^{[\infty]}$ vs time]{\includegraphics[width=0.45\linewidth]{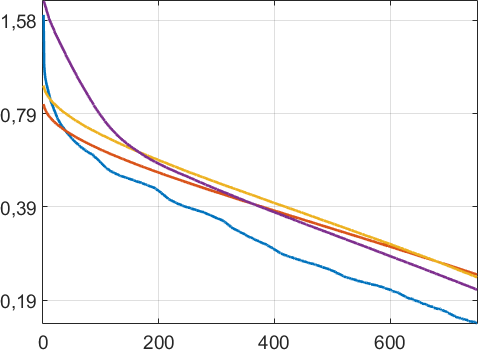}}
\caption{Comparison of training performance (time is expressed in seconds).}
\label{fig:comparison}
\end{figure}

\section{Conclusion}\label{sec:five}
In this paper, we have proposed a block-coordinate Douglas-Rachford algorithm for sparse logistic regression. In contrast to gradient-like methods, our approach relies on the proximity operator of the logistic loss, for which we derived a closed-form expression that can be efficiently implemented. Thanks to this feature, our approach removes restrictions on the choice of the algorithm parameters, unlike gradient-like methods, for which it is essential that the learning rate is carefully chosen. This is confirmed by our numerical results, which indicate that the training performance of the proposed algorithm compares favorably with state-of-the-art stochastic methods. 

\bibliographystyle{IEEEtran}
\bibliography{bib/IEEEabrv,bib/refs}

\end{document}